\documentclass[12pt]{article}
\usepackage{amsmath,amsfonts,amssymb,amsthm,txfonts,hyperref,url,t1enc}
\newtheorem{lemma}{Lemma}
\newtheorem{theorem}{Theorem}
\newtheorem{question}{Question}
\newcommand{\N}{\mathbb N}
\newcommand{\Z}{\mathbb Z}
\newcommand{\K}{\textbf{\textit{K}}}
\author{Apoloniusz Tyszka}
\title{\large{Is there an algorithm which takes as input a Diophantine
equation, returns an integer, and this integer is greater than
the number of integer solutions, if the solution set is finite?}}
\date{}
\begin{document}
\def\UrlFont{\em}
\begin{sloppypar}
\maketitle
\begin{abstract}
Let \mbox{$E_n=\{x_i=1,~x_i+x_j=x_k,~x_i \cdot x_j=x_k: i,j,k \in \{1,\ldots,n\}\}$}.
For a positive integer $n$, let \mbox{$f(n)$} denote the greatest finite total
number of solutions of a subsystem of $E_n$ in integers \mbox{$x_1,\ldots,x_n$}.
We prove: {\tt (1)}~the function~$f$ is strictly increasing, {\tt (2)}~if a non-decreasing
function $g$ from positive integers to positive integers satisfies \mbox{$f(n) \leq g(n)$}
for any $n$, then a \mbox{finite-fold} Diophantine representation of $g$ does not exist,
{\tt (3)}~if the question of the title has a positive answer, then there is a computable
strictly increasing function~$g$ from positive integers to positive integers such that
\mbox{$f(n) \leq g(n)$} for any $n$ and a \mbox{finite-fold} Diophantine representation
of $g$ does not exist.
\end{abstract}
\vskip 0.5truecm
\par
The Davis-Putnam-Robinson-Matiyasevich theorem states that every recursively enumerable
set \mbox{${\cal M} \subseteq {\N}^n$} has a Diophantine representation, that is
\begin{equation}
\tag*{\tt (R)}
(a_1,\ldots,a_n) \in {\cal M} \Longleftrightarrow
\exists x_1, \ldots, x_m \in \N ~~W(a_1,\ldots,a_n,x_1,\ldots,x_m)=0
\end{equation}
for some polynomial $W$ with integer coefficients, see \cite{Matiyasevich1}.
The polynomial~$W$ can be computed, if we know a Turing machine~$M$
such that, for all \mbox{$(a_1,\ldots,a_n) \in {\N}^n$}, $M$ halts on \mbox{$(a_1,\ldots,a_n)$} if and only if
\mbox{$(a_1,\ldots,a_n) \in {\cal M}$}, see \cite{Matiyasevich1}.
\vskip 0.2truecm
\par
The representation~{\tt (R)} is said to be \mbox{finite-fold} if for any \mbox{$a_1,\ldots,a_n \in \N$}
the equation \mbox{$W(a_1,\ldots,a_n,x_1,\ldots,x_m)=0$} has only finitely many solutions
\mbox{$(x_1,\ldots,x_m) \in {\N}^m$}.
\vskip 0.2truecm
\noindent
{\bf Open Problem}~(\mbox{\cite[pp.~341--342]{DMR}}, \mbox{\cite[p.~42]{Matiyasevich2}}, \mbox{\cite[p.~79]{Matiyasevich3}}).
{\em Does each recursively enumerable set \mbox{${\cal M} \subseteq {\N}^n$} has a \mbox{finite-fold} Diophantine representation?}
\vskip 0.3truecm
\par
\mbox{Yu. Matiyasevich} conjectures that the answer is yes. Let us pose the following two questions:
\begin{question}\label{que1}
Is there an algorithm which takes as input a Diophantine equation, returns an integer, and this
integer is greater than the modulus of integer solutions, if the solution set is finite?
\end{question}
\begin{question}\label{que2}
Is there an algorithm which takes as input a Diophantine equation, returns an integer, and this
integer is greater than the number of integer solutions, if the solution set is finite?
\end{question}
\par
Obviously, a positive answer to Question~\ref{que1} implies a positive answer to Question~\ref{que2}.
A positive answer to Question~\ref{que1} contradicts Matiyasevich's conjecture,
see \cite[p.~42]{Matiyasevich2}.
We are going to show that a positive answer to Question~\ref{que2} also contradicts Matiyasevich's conjecture.
\vskip 0.2truecm
\par
Let \mbox{$\cal{R}${\sl ng}} denote the class of all rings $\K$ that extend $\Z$, and let
\[
E_n=\{x_i=1,~x_i+x_j=x_k,~x_i \cdot x_j=x_k: i,j,k \in \{1,\ldots,n\}\}
\]
\begin{lemma}\label{lem1}
(\mbox{\cite[p.~720]{Tyszka}}) Let \mbox{$D(x_1,\ldots,x_p) \in {\Z}[x_1,\ldots,x_p]$}.
Assume that \mbox{$d_i={\rm deg}(D,x_i) \geq 1$} for each \mbox{$i \in \{1,\ldots,p\}$}. We can compute a positive
integer \mbox{$n>p$} and a system \mbox{$T \subseteq E_n$} which satisfies the following two conditions:
\vskip 0.2truecm
\noindent
{\tt Condition 1.} If \mbox{$\K \in {\cal R}{\sl ng} \cup \{\N,~\N \setminus \{0\}\}$}, then
\[
\forall \tilde{x}_1,\ldots,\tilde{x}_p \in \K ~\Bigl(D(\tilde{x}_1,\ldots,\tilde{x}_p)=0 \Longleftrightarrow
\]
\[
\exists \tilde{x}_{p+1},\ldots,\tilde{x}_n \in \K ~(\tilde{x}_1,\ldots,\tilde{x}_p,\tilde{x}_{p+1},\ldots,\tilde{x}_n) ~solves~ T\Bigr)
\]
{\tt Condition 2.} If \mbox{$\K \in {\cal R}{\sl ng} \cup \{\N,~\N \setminus \{0\}\}$}, then
for each \mbox{$\tilde{x}_1,\ldots,\tilde{x}_p \in \K$} with \mbox{$D(\tilde{x}_1,\ldots,\tilde{x}_p)=0$},
there exists a unique tuple \mbox{$(\tilde{x}_{p+1},\ldots,\tilde{x}_n) \in {\K}^{n-p}$} such that the tuple
\mbox{$(\tilde{x}_1,\ldots,\tilde{x}_p,\tilde{x}_{p+1},\ldots,\tilde{x}_n)$} \mbox{solves $T$}.
\vskip 0.2truecm
\noindent
Conditions 1 and 2 imply that for each \mbox{$\K \in {\cal R}{\sl ng} \cup \{\N,~\N \setminus \{0\}\}$}, the equation
\mbox{$D(x_1,\ldots,x_p)=0$} and the \mbox{system $T$} have the same number of solutions \mbox{in $\K$}.
\end{lemma}
\par
For a positive integer $n$, let \mbox{$f(n)$} denote the greatest finite total number of solutions of
a subsystem of $E_n$ in integers \mbox{$x_1,\ldots,x_n$}. Obviously, \mbox{$f(1)=2$} as the equation
\mbox{$x_1 \cdot x_1=x_1$} has exactly two integer solutions.
\begin{lemma}\label{lem2}
For each positive integer $n$, \mbox{$f(n+1) \geq 2 \cdot f(n)>f(n)$}.
\end{lemma}
\begin{proof}
If $r$ is a positive integer and a system \mbox{$S \subseteq E_n$} has exactly $r$
solutions in integers \mbox{$x_1,\ldots,x_n$}, then the system
\mbox{$S \cup \{x_{n+1} \cdot x_{n+1}=x_{n+1}\} \subseteq E_{n+1}$} has exactly
\mbox{$2r$} solutions in integers \mbox{$x_1,\ldots,x_{n+1}$}.
\end{proof}
\noindent
{\bf Corollary.} {\em The function $f$ is strictly increasing.}
\vskip 0.2truecm
\par
\mbox{A function} \mbox{$\beta: \N \setminus \{0\} \to \N \setminus \{0\}$} is said to majorize a function
\mbox{$\alpha: \N \setminus \{0\} \to \N \setminus \{0\}$} provided \mbox{$\alpha(n) \leq \beta(n)$}
for any $n$.
\begin{theorem}\label{the1}
If a non-decreasing function \mbox{$g:\N \setminus \{0\} \to \N \setminus \{0\}$} majorizes $f$,
then a \mbox{finite-fold} Diophantine representation of $g$ does not exist.
\end{theorem}
\begin{proof}
Assume, on the contrary, that there is a \mbox{finite-fold} Diophantine representation of $g$.
It means that there is a polynomial \mbox{$W(x_1,x_2,x_3,\ldots,x_m)$} with integer coefficients
such that
\begin{description}
\item{{\tt (1)~}}
for any \mbox{non-negative} integers~\mbox{$x_1,x_2$},
\[
(x_1,x_2) \in g \Longleftrightarrow \exists x_3,\ldots,x_m \in \N ~~W(x_1,x_2,x_3,\ldots,x_m)=0
\]
\end{description}
and for each \mbox{non-negative} integers \mbox{$x_1,x_2$} at most finitely many tuples
\mbox{$(x_3,\ldots,x_m) \in {\N}^{m-2}$} satisfy \mbox{$W(x_1,x_2,x_3,\ldots,x_m)=0$}.
By Lemma~\ref{lem1} for \mbox{$\K=\N$}, there is a formula \mbox{$\Phi(x_1,x_2,x_3,\ldots,x_s)$} such that
\begin{description}
\item{{\tt (2)~}}
\mbox{$s \geq {\rm max}(m,3)$} and \mbox{$\Phi(x_1,x_2,x_3,\ldots,x_s)$} is a conjunction of formulae of the forms
\mbox{$x_i=1$}, \mbox{$x_i+x_j=x_k$}, \mbox{$x_i \cdot x_j=x_k$} \mbox{$(i,j,k \in \{1,\ldots,s\})$} which equivalently
expresses that \mbox{$W(x_1,x_2,x_3,\ldots,x_m)=0$} and each \mbox{$x_i~(i=1,\ldots,m)$} is a sum of four squares.
\end{description}
Let $S$ denote the following system
\[\left\{
\begin{array}{rcl}
a \cdot a &=& A \\
b \cdot b &=& B \\
c \cdot c &=& C \\
d \cdot d &=& D \\
A+B &=& u_1 \\
C+D &=& u_2 \\
u_1+u_2 &=& u_3 \\
\tilde{a} \cdot \tilde{a} &=& \tilde{A} \\
\tilde{b} \cdot \tilde{b} &=& \tilde{B} \\
\tilde{c} \cdot \tilde{c} &=& \tilde{C} \\
\tilde{d} \cdot \tilde{d} &=& \tilde{D} \\
\tilde{A}+\tilde{B} &=& \tilde{u}_1 \\
\tilde{C}+\tilde{D} &=& \tilde{u}_2 \\
\tilde{u}_1+\tilde{u}_2 &=& \tilde{u}_3 \\
u_3+\tilde{u}_3 &=& x_2 \\
t_1 &=& 1 \\
t_1+t_1 &=& t_2 \\
t_2 \cdot t_2 &=& t_3 \\
t_3 \cdot t_3 &=& t_4 \\
&\ldots& \\
t_{s-1} \cdot t_{s-1} &=& t_s \\
t_s \cdot t_s &=& t_{s+1} \\
t_{s+1} \cdot t_{s+1} &=& x_1 \\
{\rm all~equations~occurring~in~}\Phi(x_1,x_2,x_3,\ldots,x_s) \\
\end{array}
\right.\]
with \mbox{$2s+23$} variables. The system $S$ equivalently expresses the following conjunction:
\[
\Biggl(\left(a^2+b^2+c^2+d^2\right)+\left({\tilde{a}}^2+{\tilde{b}}^2+{\tilde{c}}^2+{\tilde{d}}^2\right)=x_2\Biggr) \wedge
\Biggl(x_1=2^{\textstyle 2^s}\Biggr) \wedge \Phi(x_1,x_2,x_3,\ldots,x_s)
\]
Conditions {\tt (1)-(2)} and Lagrange's four-square theorem imply that the system~$S$ is satisfiable over
integers and has only finitely many integer solutions. Let $L$ denote the number of integer solutions to $S$.
If an integer tuple solves $S$, then \mbox{$x_1=2^{\textstyle 2^s}$} and \mbox{$x_2=g(x_1)=g\left(2^{\textstyle 2^s}\right)$}.
Since the equation \mbox{$u_3+\tilde{u}_3=x_2$} belongs to $S$ and Lagrange's four-square theorem holds,
\mbox{$L \geq g\left(2^{\textstyle 2^s}\right)+1$}. The definition of~$f$ implies that
\begin{equation}
\tag*{{\tt (3)}}
L \leq f\left(2s+23\right)
\end{equation}
Since $g$ majorizes $f$,
\begin{equation}
\tag*{{\tt (4)}}
f\left(2s+23\right)<g\left(2s+23\right)+1
\end{equation}
Since \mbox{$s \geq 3$} and $g$ is non-decreasing,
\begin{equation}
\tag*{{\tt (5)}}
g\left(2s+23\right)+1 \leq g\left(2^{\textstyle 2^s}\right)+1
\end{equation}
Inequalities {\tt (3)-(5)} imply that \mbox{$L<g\left(2^{\textstyle 2^s}\right)+1$}, a contradiction.
\end{proof}
\begin{theorem}\label{the2}
If Question~\ref{que2} has a positive answer, then there is a computable strictly increasing
function \mbox{$g:\N \setminus \{0\} \to \N \setminus \{0\}$} such that $g$ majorizes $f$ and
a \mbox{finite-fold} Diophantine representation of $g$ does not exist.
\end{theorem}
\begin{proof}
For each positive integer $r$, there are only finitely many Diophantine equations whose
lengths are not greater than $r$, and these equations can be algorithmically constructed.
This and the assumption that the question of the title has a positive answer imply that
there exists a computable function \mbox{$\delta: \N \setminus \{0\} \to \N \setminus \{0\}$}
such that for each positive integer $r$ and for each Diophantine equation whose length is not
greater than $r$, $\delta(r)$ is greater than the number of integer solutions if the solution set is finite.
There is a computable function \mbox{$\psi:\N \setminus \{0\} \to \N \setminus \{0\}$}
such that each subsystem of $E_n$ is equivalent to a Diophantine equation whose length
is not greater than $\psi(n)$. The function
\[
\N \setminus \{0\} \ni n \stackrel{\textstyle h}{\longmapsto} \delta(\psi(n)) \in \N \setminus \{0\}
\]
is computable. The definition of $f$ implies that $h$ majorizes $f$. The function
\[
\N \setminus \{0\} \ni n \stackrel{\textstyle g}{\longmapsto} \sum_{i=1}^n h(i) \in \N \setminus \{0\}
\]
is computable and strictly increasing. Since $g$ majorizes $h$ and $h$ majorizes $f$, \mbox{$g$ majorizes $f$}.
By Theorem~\ref{the1}, a \mbox{finite-fold} Diophantine representation of $g$ does not exist.
\end{proof}

\noindent
Apoloniusz Tyszka\\
University of Agriculture\\
Faculty of Production and Power Engineering\\
Balicka 116B, 30-149 Krak\'ow, Poland\\
E-mail address: \url{rttyszka@cyf-kr.edu.pl}
\end{sloppypar}
\end{document}